\documentclass{article}

\usepackage[utf8]{inputenc}
\usepackage[T1]{fontenc}
\usepackage[margin=1.25in]{geometry}
\usepackage{amsmath, amssymb, amsthm, amsfonts}
\usepackage{mathtools}
\usepackage{bbm}
\usepackage{xcolor}
\usepackage{hyperref}
\usepackage[shortlabels]{enumitem}

\newtheorem{theorem}{Theorem}
\newtheorem{proposition}[theorem]{Proposition}
\newtheorem{lemma}[theorem]{Lemma}

\newtheorem{definition}[theorem]{Definition}

\newcommand{\reals}{\mathbb{R}}
\newcommand{\nats}{\mathbb{N}}
\newcommand{\ex}{\mathbb{E}}
\newcommand{\pr}{\mathbb{P}}
\newcommand{\var}{\operatorname{Var}}
\newcommand{\norm}[1]{\|#1\|}
\newcommand{\infnorm}[1]{\|#1\|_\infty}

\newcommand{\Lip}{\operatorname{Lip}}
\newcommand{\ind}{\mathbbm{1}}
\newcommand{\Lset}{\mathcal{L}}

\title{Wasserstein Concentration of Empirical Measures for Dependent Data via the Method of Moments}

\author{
  Arash A. Amini \\
  \texttt{aaamini@g.ucla.edu}
  \and
  Luciano Vinas \\
  \texttt{lucianovinas@g.ucla.edu}
}

\date{\today}

\begin{document}

\maketitle

\begin{abstract}
We establish a general concentration result for the 1-Wasserstein distance between the empirical measure of a sequence of random variables and its expectation. Unlike standard results that rely on independence (e.g., Sanov's theorem) or specific mixing conditions, our result requires only two conditions: (1) control over the variance of the empirical moments, and (2) a flexible tail condition we term $\Psi_{r_n}$-sub-Gaussianity. This approach allows for significant dependencies between variables, provided their algebraic moments behave predictably. The proof uses the method of moments combined with a polynomial approximation of Lipschitz functions via Jackson kernels, allowing us to translate moment concentration into topological concentration.
\end{abstract}

\section{Introduction}

The convergence of the empirical measure $\mu_n = \frac{1}{n}\sum_{i=1}^n \delta_{Y_i}$ to a deterministic limit $\bar{\mu}$ is a fundamental problem in probability and statistics. When the variables $Y_i$ are independent and identically distributed (i.i.d.), classical results such as Sanov's theorem or the Varadarajan theorem provide strong guarantees. However, in many modern high-dimensional settings---such as the spectral distribution of random matrices, interacting particle systems, or feature propagation in graph neural networks---the variables $Y_i$ exhibit complex dependency structures.

In such dependent settings, establishing independence or mixing properties can be intractable. However, it is often possible to compute the \emph{moments} of the system, $\ex[n^{-1}\sum Y_i^k]$, via combinatorial techniques (e.g., trace methods or walk counting). The classical \emph{Method of Moments}\cite[Section 2.5]{Vaart98} uses these algebraic quantities to prove weak convergence.

In this note, we provide a ``glue'' theorem that upgrades algebraic moment control to concentration in the 1-Wasserstein metric ($W_1$). Specifically, we show that if the variance of the empirical moments vanishes asymptotically, and the tails of the distribution decay sufficiently fast, then the empirical measure concentrates in $W_1$.

A key technical contribution is the handling of the test functions. While moment convergence naturally controls polynomials, the $W_1$ metric is defined via Lipschitz functions. To bridge this gap, 
we employ a polynomial approximation of Lipschitz functions using Chebyshev-Jackson kernels. This allows us to control the size of the coefficients of the polynomial, while achieving the optimal rate of approximation for Lipschitz functions, both of which are necessary for the proof.

\section{Preliminaries}

Let $\mu_n = \frac{1}{n} \sum_{i=1}^n \delta_{Y_{i,n}}$ be the empirical measure of a triangular array of real-valued random variables $\{Y_{i,n}\}_{i=1}^n$. We denote the expected measure as $\bar{\mu}_n = \ex [\mu_n]$.
The 1-Wasserstein distance between two probability measures $\mu$ and $\nu$ on $\reals$ is defined by the dual form:
\begin{equation}
    W_1(\mu, \nu) = \sup_{f \in \Lip(1)} \left| \int f d\mu - \int f d\nu \right|,
\end{equation}
where $\Lip(1)$ denotes the set of $1$-Lipschitz functions $f: \reals \to \reals$.

\subsection{The $\Psi_{r}$ Norm and Asymptotic Sub-Gaussianity}

A central feature of our analysis is a tail condition that is milder than uniform sub-Gaussianity. We consider a sequence of parameters $r_n$ that may grow with $n$.

\begin{definition}[Truncated $\Psi_r$ Norm]
    For a real number $r \ge 2$, let $\Psi_r(x) = \sum_{j=1}^{\lfloor r/2 \rfloor} \frac{x^{2j}}{j!}$. The $\Psi_r$ norm of a random variable $X$ is defined as:
    \begin{equation}
        \|X\|_{\Psi_r} = \inf \left\{ K > 0 : \ex \left[\Psi_r(|X|/K)\right] \le 1 \right\}.
    \end{equation}
\end{definition}

If $r \to \infty$, $\Psi_r(x)$ approaches $\exp(x^2) - 1$, recovering the standard $\psi_2$ (sub-Gaussian) norm. However, for finite $r$, $\|X\|_{\Psi_r} < \infty$ only implies that moments up to order $r$ behave like those of a sub-Gaussian variable. This allows us to work with variables that may not be strictly sub-Gaussian for fixed $n$, provided their moments are well-behaved up to a growing order $r_n$.

\begin{lemma}[Moments and Tails]
\label{lem:psi_properties}
    Let $r \ge 2$ and $\|X\|_{\Psi_r} \le K$. Then:
    \begin{enumerate}[(a)]
        \item (Moment Bound) For all $p \in [2, 2\lfloor r/2 \rfloor]$, $(\ex |X|^p)^{1/p} \le C_1 K \sqrt{p}$.
        \item (Tail Bound) There exist constants $c_0, c_1 > 0$ such that for $t \ge c_0 K$:
        $$ \pr(|X| \ge t) \le \exp\left(-c_1 \min\left\{ \frac{t^2}{K^2}, \lfloor r/2 \rfloor \right\}\right). $$
    \end{enumerate}
\end{lemma}
\noindent The proof of part (a) follows standard Orlicz norm arguments. For part~(b), see~\cite[Lemma 25]{Vinas24}.

\section{Main Results}

Our main result establishes $W_1$ concentration based on moment variance and tail decay.

\begin{theorem}
\label{thm:main_result}
    Let $\{Y_{i,n}\}_{i=1}^n$ be a triangular array of real-valued random variables. Let $\mu_n$ be their empirical measure and $\bar{\mu}_n = \ex[\mu_n]$. Assume there exists a sequence $r_n = \omega(1)$ such that:
    \begin{enumerate}[label=(\alph*)]
        \item \textbf{Uniform Tail Control:} The variables are uniformly $\Psi_{r_n}$ sub-Gaussian. That is, there exists $\zeta > 0$ such that
        $
       \sup_{i\in [n]} \norm{Y_{i,n}}_{\Psi_{r_n}} \le \zeta.
        $
        \item \textbf{Moment Concentration:} For every fixed integer $k \ge 1$, the variance of the empirical moments vanishes:
        $$ \var\left(\frac{1}{n}\sum_{i=1}^n Y_{i,n}^k\right) \to 0 \quad \text{as } n \to \infty. $$
    \end{enumerate}
    Then, the expected Wasserstein distance vanishes:
    $$ \ex \bigl[ W_1(\mu_n, \bar \mu_n) \bigr] \to 0 \quad \text{as } n \to \infty. $$
\end{theorem}

    Condition (b) is satisfied trivially if the variables are i.i.d., but it also holds for many dependent systems where correlations decay sufficiently fast (e.g., spectral statistics of random matrices, node features in sparse graphs). Condition (a) ensures that while the variables need not be bounded, their mass does not escape to infinity too quickly. 

\subsection{Extension to Random Vectors}

The method of moments is particularly powerful in high dimensions where computing projected moments $\ex \langle \theta, Y \rangle^k$ is often feasible. We extend Theorem \ref{thm:main_result} to random vectors in $\reals^d$.

We denote the absolute first moment of a measure $\mu$ on $\reals^d$ by $M_1(\mu) = \int \|x\| d\mu(x)$.

\begin{theorem}
\label{thm:vector_result}
    Let $Y_{i,n}$ be random vectors in $\reals^d$. Let $\mu_n = \frac1n \sum_{i=1}^n \delta_{Y_{i,n}}$ and $\bar \mu_n = \ex \mu_n$. Assume there exists a sequence $r_n = \omega(1)$ such that for every unit vector $\theta \in S^{d-1}$:
    \begin{enumerate}[label=(\alph*)]
        \item \textbf{Uniform Projected Tails:} The projections are uniformly $\Psi_{r_n}$ sub-Gaussian:
            $$ \sup_{i\in [n]} \norm{\langle\theta, Y_{i,n}\rangle}_{\Psi_{r_n}} \le \zeta(\theta) < \infty. $$
        \item \textbf{Projected Moment Concentration:} For every fixed $k \ge 1$:
            $$ \var\left(\frac{1}{n}\sum_{i=1}^n \langle\theta, Y_{i,n}\rangle^k\right) \to 0 \quad \text{as } n\to\infty. $$
        \item \textbf{Bounded First Moment:} $\sup_{n\ge 1} M_1(\bar \mu_n) < \infty$.
    \end{enumerate}
    Then, $\ex \bigl[ W_1(\mu_n, \bar \mu_n) \bigr] \to 0$ as $n\to\infty$.
\end{theorem}

\section{Polynomial approximations}

To prove the main result, we require results on polynomial approximation of Lipschitz functions. Standard Weierstrass approximation yields a rate of $m^{-1/2}$ for approximation by polynomials of degree $m$, which is too slow. The optimal rate of approximation for Lipschitz functions is $m^{-1}$, which will be sufficient. However, we also need to control the size of the coefficients of the polynomial.
We use the Chebyshev--Jackson approximation, which achieves this balance.

We begin with a bound on the coefficients of Chebyshev polynomials.

\begin{lemma} 
\label{lem:chebyshev_coeffs}
    Let $T_k$ be the $k$th Chebyshev polynomial, and let $[T_k]_j$ be the coefficient of $x^j$ in $T_k(x)$. Then, $|[T_k]_0| \le 1$ and
    \begin{align*}
        \max_{1 \le j \le k} |[T_k]_j| \le (1+\sqrt 2)^k \le 3^k.
    \end{align*}
\end{lemma}
\begin{proof}
    The first part is clear, since $[T_k]_0 \in \{0,1\}$. For the second part, from the recurrence relation $T_{k+1}(x) = 2x T_k(x) - T_{k-1}(x)$, we have
    \begin{align*}
        |[T_{k+1}]_j| \le 2|[T_k]_{j-1}| + |[T_{k-1}]_{j}|.
    \end{align*}
    Assuming the result holds as $ \max_{1 \le j \le k} |[T_k]_j| \le c^k$ for some constant $c$ and  
    for all $T_r, r \le k$, we have $|[T_{k+1}]_j| \le 2 \cdot c^{k} + c^{k-1} $. Then, if $2 c^k + c^{k-1} \le c^{k+1}$, the result follows by induction. This inequality holds for $c \ge 1+\sqrt{2}$. The proof is complete.
\end{proof}
    
\begin{lemma}[Chebyshev--Jackson approximation]
\label{lem:chebyshev_approx}
    Let $B \ge 3$. Then, for any $f: [-B,B] \to \mathbb{R}$ $1$-Lipschitz with $f(0) = 0$, there exists a polynomial $P(x) = \sum_{j=0}^m c_j x^j$, with $m \in 4 \nats$, such that
    \begin{align*}
        \sup_{x \,\in\, [-B,B]} |f(x) - P(x)| \le \frac{18 B}{m}, \quad \text{and} \quad |c_j| \le 6B \cdot 3^{m-j}, \quad \text{for all } j \ge 0.
    \end{align*}
\end{lemma}
    
\begin{proof}
    Consider an $L$-Lipschitz function $g$ on $[-1,1]$ with $g(0) = 0$. Then, for each $m \in 4\nats$, there is a polynomial of the form
    \[
        Q_m(x) = \sum_{k=0}^m \lambda_{k,m} a_k(g) T_k(x)
    \]
    where $\lambda_{k,m}$ are derived from a Jackson kernel, satisfying $0 \le \lambda_{k,m} \le 1$ and $a_k(g)$ are the Chebyshev coefficients of $g$, such that
    \[
    \sup_{x \in [-1,1]} |g(x) - Q_m(x)| \le \frac{18 L}{m}, \quad |a_k(g)| \le \frac{\sqrt{8/\pi} L}{k}, \quad k \ge 1.
    \]
    See Facts~3.2 and 3.3 in~\cite{braverman2022}.
    The Chebyshev coefficients are given by
    \[
        a_k(g) = \frac{2}{\pi} \int_{-1}^1 \frac{g(x) T_k(x)}{\sqrt{1 - x^2}} dx, \quad k \ge 1,
    \]
    and for $k = 0$, the same formula holds with $2/\pi$ replaced with $1/\pi$. For $k = 0$, using $g(0) = 0$ so that $|g(x)| \le L |x|$ for all $x \in [-1,1]$, and $T_0(x)=1$, we have
    \begin{align*}
        |a_0(g)| &\le \frac1\pi \int_{-1}^1 \frac{L |x|}{\sqrt{1-x^2}}dx = \frac{2L}\pi.
    \end{align*}
    Thus a crude upper bound that works for all $k \ge 0$ is $|a_k(g)| \le 2L$.

    Let $a_{k,m} = \lambda_{k,m} a_k(g)$ and note that $|a_{k, m}| \le 2L$ for all $k \ge 0$, by the above discussion.
    Rewriting $Q_m(x) = \sum_{j=0}^m b_j x^j$, one has
    $
    b_j = \sum_{k=j}^m a_{k,m}  [T_k]_j
    $
    where $[T_k]_j$ is the coefficient of $x^j$ in $T_k(x)$. 
    It follows from Lemma \ref{lem:chebyshev_coeffs} that
    \begin{align*}
        |b_j| \le \sum_{k=j}^m 2 L \cdot 3^k \le 2L \cdot 3^m \sum_{k=j}^m 3^{k-m} \le  2L \cdot 3^m \frac{1}{1-3^{-1}} \le 6 L \cdot 3^m
    \end{align*}
    for all $j \ge 0$. 

    If $f$ is 1-Lipschitz on $[-B,B]$ with $f(0) = 0$, then $g(x) = f(Bx)$ is $B$-Lipschitz on $[-1,1]$ with $g(0) = 0$. Let $Q_m$ be the above polynomial for $g$, and let $P(x) = Q_m(x/B) = \sum_{j=0}^m (b_j / B^j) x^j =: \sum_{j=0}^m c_j x^j$.  Then,
    \[
    |c_j| \le 6 B \frac{3^m}{B^j} \le 6B \cdot 3^{m-j}
    \]
    assuming $B \ge 3$. We also have $\sup_{x \in [-B,B]} |f(x) - P(x)|  = \sup_{x \in [-1,1]} |g(x) - Q_m(x)| \le \frac{18B}{m}$. The proof is complete.
\end{proof}

\section{Proof of Main Result (Scalar Case)}

\begin{proof}[Proof of Theorem \ref{thm:main_result}]
     Let us write $\Lip(f) = \sup_{x \neq y} \frac{|f(x)-f(y)|}{|x-y|}$ for the Lipschitz constant of $f$. Consider the set of functions
    \[
    \Lset = \{f: \reals \to \reals \mid \Lip(f) \leq 1, f(0) = 0\}, \quad 
    \Lset_B =  \{f \ind_{|x| \le B} \mid f \in \Lset,\; B > 0\}.
    \]
    Let $\varpi_n := \mu_n - \bar \mu_n$. By the dual characterization of $W_1$, we have
    \[
    W_1(\mu_n, \bar \mu_n) \le \sup_{f \in \Lset} |\varpi_n f|.
    \]
    By breaking $f = f \ind_{|x| \le B} + f \ind_{|x| > B}$, we have
    \begin{align}\label{eq:wass_conv_break}
        W_1(\mu_n, \bar \mu_n) &\le  \sup_{f \in \Lset_B} |\varpi_n f | + \sup_{f \in \Lset} |\varpi_n (f \ind_{|x| > B})|.
    \end{align}
    Fix $\epsilon \in (0,1)$ and consider the second term first.
    For any integrable $f$, we have
    \begin{align}
        |\varpi_n(f \ind_{|x| > B})| &\le |\mu_n(f \ind_{|x| > B})| + |\bar \mu_n(f \ind_{|x| > B})| \notag \\
        &\le \mu_n (| f | \ind_{|x| > B}) + \bar \mu_n (| f | \ind_{|x| > B}). \label{eq:wass_diff_f_tail}
    \end{align}
    For $f \in \Lset$, we have $|f(x)| = |f(x) - f(0)| \le |x - 0|$. 

    Then, we have
    \begin{align*}
        |\varpi_n(f \ind_{|x| > B})| &\le \mu_n (| x | \ind_{|x| > B}) + \bar \mu_n (| x | \ind_{|x| > B})
    \end{align*}
    Taking the supremum over $f \in \Lset$ and then expectation, we have
    \begin{align*}
        \ex \sup_{f \in \Lset} |\varpi_n(f \ind_{|x| > B})| 
        &\;\le\; 2 \bar \mu_n( |x| \ind_{|x| > B}) = \frac2{n} \sum_{i=1}^n \ex \bigl( |Y_{i,n}| \ind_{\{|Y_{i,n}| > B\}} \bigr).
    \end{align*}
    Take $n$ large enough so that 
    \begin{equation}\label{eq:rn:lower:1}
        r_n \ge 2\Bigl(\frac{B^2}{\zeta^2}+1\Bigr)
    \end{equation}
    which we will verify at the end. Also, take $B \ge B_0(\zeta) := c_0 \zeta$ where $c_0$ is the constant in Lemma~\ref{lem:psi_properties}(b). Then, by this lemma, we have $\pr(|Y_{i,n}| > B) \le \exp(-c_1 B^2/\zeta^2)$, and by Lemma~\ref{lem:psi_properties}(a), we have 
    $\ex[Y_{i,n}^2] \le 2 C_1^2 \zeta^2$. Then, by Cauchy-Schwarz, we have
    \begin{align*}
        \ex \bigl( |Y_{i,n}| \ind_{\{|Y_{i,n}| > B\}} \bigr) \le \sqrt{\ex[|Y_{i,n}|^2] \cdot \pr(|Y_{i,n}| > B)} \le \sqrt{2} C_1 \zeta \cdot \exp(-c B^2/2\zeta^2).
    \end{align*}
    Taking $B \ge B_1(\zeta)$ for $B_1(\zeta)$ large enough, the RHS can be made $\le \epsilon$, which gives
    \begin{align*}
        \ex \sup_{f \in \Lset} |\varpi_n(f \ind_{|x| > B})|  \le 2\epsilon.
    \end{align*}

    Consider now the first term in~\eqref{eq:wass_conv_break}.
    Viewing $\Lset_B$ as a subspace of $(C_b([-B,B]), \|\cdot\|_\infty)$, by restricting to $[-B,B]$, $\Lset_B$ is uniformly bounded and equicontinuous, hence by Arzel\`{a}--Ascoli, it is relatively compact in the sup-norm topology. This, in turn, implies $\Lset_B$ is totally bounded.  Then, there exists $f_1, \ldots, f_M \in \Lset_B$ that form an $\epsilon$-net for $\Lset_B$ in sup-norm, for some $M = M(\epsilon,B) < \infty$. That is, for any $f \in \Lset_B$, there is $f_\ell$ such that $\infnorm{f-f_\ell} \leq \epsilon$, hence
    \begin{align*}
    |\varpi_n f| &\;\le\; |\varpi_n(f-f_\ell)| + |\varpi_n f_\ell|\\
    &\;\le\; \norm{\varpi_n}_{\text{TV}} \cdot \infnorm{f-f_\ell} + |\varpi_n f_\ell|
    \;\le\; 2 \epsilon + |\varpi_n f_\ell|. 
    \end{align*}
    Taking supremum over $f \in \Lset_B$, we have
    \begin{align*}
    \sup_{f \in \Lset_B} |\varpi_nf| 
    &\;\le\; 2 \epsilon + \sup_{\ell \in [M]} |\varpi_n f_\ell|. 
    \end{align*}
    
    Take $B \ge 3$. By Lemma~\ref{lem:chebyshev_approx}, each $f_\ell$ admits a (truncated) polynomial $Q_{\ell}(x) = \ind_{\{|x| \leq  B\}}\cdot \sum_{j=0}^{m} c_{j\ell} x^j $, with $m = 4 \lceil C_2 B / \epsilon \rceil \in 4\nats$ (can take $C_2=18$) such that
    \[
    \infnorm{f_\ell - Q_\ell} \le \epsilon,
    \]
    and $|c_{j\ell}| \le 6B \cdot 3^{m-j} =: a_j$ for all $j \ge 0$ and $\ell \in [M]$. We have 
    \[
    |\varpi_n f_\ell| \le \norm{\varpi_n}_{\text{TV}} \cdot \infnorm{f_\ell - Q_\ell} + |\varpi_n Q_\ell|.
    \]
    It follows that 
    \[
    \sup_{\ell \in [M]}  |\varpi_n f_\ell| \le 2\epsilon + \sup_{\ell \in [M]} |\varpi_n Q_\ell|
    \]
    and we have
    \begin{align*}
        \sup_{\ell \in [M]} |\varpi_n Q_\ell| &\le 
        \sup_{\ell \in [M]} \Bigl| \sum_{j=0}^{m} c_{j\ell}\,  \varpi_n(x^j \ind_{|x| \le B}) \Bigr| \\
        &\le \sum_{j=0}^{m} \bigl(\sup_{\ell \in [M]} |c_{j\ell}| \bigr) \cdot |\varpi_n(x^j \ind_{|x| \le B})|  \le \sum_{j=0}^{m} a_j  \,|\varpi_n(x^j \ind_{|x| \le B})|
    \end{align*}    
    We have 
    \begin{align*}
        | \varpi_n(x^j \ind_{|x| \le B}) | \le  | \varpi_n(x^j) | + | \varpi_n(x^j \ind_{|x| > B}) |.
    \end{align*}     
     Then, for the second term, using~\eqref{eq:wass_diff_f_tail}, we have, for all $j \in [m]$,
    \begin{align*}
        | \varpi_n(x^j \ind_{|x| > B}) | &\le  \mu_n(|x^j|\ind_{|x| > B}) + \bar \mu_n(|x^j|\ind_{|x| > B}) \\
        &\le  \mu_n(|x^m|\ind_{|x| > B}) + \bar \mu_n(|x^m|\ind_{|x| > B}).
    \end{align*}
    Taking maximum over $j \in [m]$, followed by expectation, we have
    \begin{align*}
        \ex \sup_{j \in [m]} | \varpi_n(x^j \ind_{|x| > B})| &\le  2 \bar \mu_n(|x|^m \ind_{|x| > B})  = \frac2{n} \sum_{i=1}^n \ex \bigl( |Y_{i,n}|^m \ind_{\{|Y_{i,n}| > B\}} \bigr).
    \end{align*}
    Take $n$ large enough so that \begin{equation}\label{eq:rn:lower:2}
        r_n \ge 2 m = 8 \lceil C_2 B / \epsilon \rceil,
    \end{equation}
    which we will verify at the end.
    Then, by Lemma~\ref{lem:psi_properties}(a) we have
    $\ex[|Y_{i,n}|^{2m}] \le (C_1 \zeta)^{2m} (2m)^{m} = (2C_1^2 \zeta^2 m)^{m}$. Then, by Cauchy-Schwarz, we have
    \begin{align*}
        \ex \bigl( |Y_{i,n}|^m \ind_{\{|Y_{i,n}| > B\}} \bigr) &\le \sqrt{\ex[|Y_{i,n}|^{2m}] \cdot \pr(|Y_{i,n}| > B)} \\ &\le (2C_1^2 \zeta^2 m)^{m} \cdot \exp(-c B^2/2\zeta^2) 
    \end{align*}
    Using $a_j = 6B \cdot 3^{m-j}$, we have $\sum_{j=0}^m a_j \le 9 B\cdot 3^m$. It follows that
    \begin{align*}
        \ex \Bigl[ \sum_{j=0}^m a_j | \varpi_n(x^j \ind_{|x| > B})| \Bigr] &\le \Bigl(\sum_{j=0}^m a_j \Bigr)\cdot \ex \sup_{j \in [m]} | \varpi_n(x^j \ind_{|x| > B})| \\
         &\le 9 B\cdot 3^m \cdot 2 (2C_1^2 \zeta^2 m)^{m} \cdot \exp(-c B^2/2\zeta^2) \\
        &\le 18 \exp \Bigl( \log B + m \log(6C_1^2 \zeta^2 m) - c B^2/2\zeta^2 \Bigr) \\
        &\le 18 \exp \biggl( \log B + 4 \lceil C_2 B / \epsilon \rceil \log\Bigl(24C_1^2 \zeta^2 \lceil C_2 B / \epsilon \rceil\Bigr) - c B^2/2\zeta^2 \biggr).
    \end{align*}
    Since $B^2$ grows faster than $B \log B$, the RHS can be made $\le \epsilon$ for $B \ge B_2(\zeta, \epsilon)$ for some $B_2(\zeta, \epsilon)$ large enough. For this choice of $B$, we have
    \begin{align*}
        \ex \sup_{\ell \in [M]} |\varpi_n Q_\ell| &\le \sum_{j=0}^m a_j \ex |\varpi_n(x^j)| + \epsilon \\
        &\le \sum_{j=0}^m a_j \sqrt{\var\Bigl(\frac1n\sum_{i=1}^n Y_{i,n}^j\Bigr)} + \epsilon.
    \end{align*}
    By assumption (b) of the theorem,
    \begin{align}\label{eq:var:less:than:eps}
    \max_{0 \le j \le m} \var\Bigl(\frac1n\sum_{i=1}^n Y_{i,n}^j\Bigr) \le \epsilon^2 / (\sum_{j=0}^m a_j)^2    
    \end{align}
     for sufficiently large $n$. This gives $\ex \sup_{\ell \in [M]} |\varpi_n Q_\ell| \le 2\epsilon$. Putting the pieces together, we have 
    \[
        \ex \sup_{f \in \Lset_B} |\varpi_n f| \le 2\epsilon + 2\epsilon + 2\epsilon = 6\epsilon.
    \]
    All in all, taking $B = \max\{3, B_0(\zeta), B_1(\zeta), B_2(\zeta, \epsilon)\}$, and $n$ large enough so that~\eqref{eq:rn:lower:1} and~\eqref{eq:rn:lower:2} are satisfied for the chosen $B$, and~\eqref{eq:var:less:than:eps} holds, we obtain $\ex W_1(\mu_n, \bar \mu_n) \le 8\epsilon$. The proof is complete.
\end{proof}

\section{Proof of Vector Extension (Theorem \ref{thm:vector_result})}

To prove the vector case, we rely on the fact that convergence of one-dimensional projections (which we established in Theorem \ref{thm:main_result}) implies convergence of the full measure, provided the first moments are controlled.

We use the notation $\eta_\theta$ to denote the pushforward of a measure $\eta$ on $\reals^d$ by the projection $x \mapsto \langle \theta, x \rangle$.

\begin{lemma}[Lipschitz Continuity of Projections]
\label{lem:proj_lip}
    Let $\eta$ be a probability measure on $\reals^d$. 
    Then, for any $\theta_1, \theta_2 \in \reals^d$,
    $$ W_1(\eta_{\theta_1}, \eta_{\theta_2}) \le \norm{\theta_1 - \theta_2} M_1(\eta). $$
\end{lemma}
\begin{proof}
    Let $X \sim \eta$. Using the dual formulation of $W_1$ for measures on $\reals$:
    \begin{align*}
        W_1(\eta_{\theta_1}, \eta_{\theta_2}) &= \sup_{f \in \Lip(1)} \bigl|\ex f(\langle \theta_1, X \rangle) - \ex f(\langle \theta_2, X \rangle)\bigr| \\
        &\le \sup_{f \in \Lip(1)}  \ex \bigl|f(\langle \theta_1, X \rangle) - f(\langle \theta_2, X \rangle)\bigr| \\
        &\le \ex |\langle \theta_1 - \theta_2, X \rangle| 
        \le \norm{\theta_1 - \theta_2} \cdot \ex \|X\|.
    \end{align*}
    This completes the proof.
\end{proof}

The following proposition allows us to upgrade scalar convergence to vector convergence.

\begin{proposition}
\label{prop:lifting}
    Let $\{\mu_n\}$ and $\{\eta_n\}$ be random probability measures on $\reals^d$ with expectations $\bar{\mu}_n$ and $\bar{\eta}_n$. Assume:
    \begin{enumerate}
        \item $\sup_{n\ge 1} (M_1(\bar \mu_n) + M_1(\bar \eta_n)) < \infty$.
        \item For every $\theta \in S^{d-1}$, $\ex \bigl[ W_1(\mu_{n,\theta}, \eta_{n,\theta}) \bigr] \to 0$ as $n \to \infty$.
    \end{enumerate}
    Then, $\ex \bigl[ W_1(\mu_n, \eta_n) \bigr] \to 0$.
\end{proposition}
\begin{proof}
    Let $L_n = M_1(\mu_n) + M_1(\eta_n)$. By Lemma \ref{lem:proj_lip} and the triangle inequality, the map $\theta \mapsto W_1(\mu_{n,\theta}, \eta_{n,\theta})$ is $L_n$-Lipschitz on the sphere. Specifically,
    \begin{align*}
       W_1(\mu_{n,\theta_1}, \eta_{n,\theta_1}) -  W_1(\mu_{n,\theta_2}, \eta_{n,\theta_2}) 
       &\le W_1(\mu_{n,\theta_1}, \mu_{n,\theta_2})  + W_1(\eta_{n,\theta_1}, \eta_{n,\theta_2}) \\
       &\le L_n \norm{\theta_1 - \theta_2}.
    \end{align*}
    
    Bayraktar and Guo \cite{bayraktar2021} proved that the $W_1$ distance in $\reals^d$ is controlled by the supremum over projections: there exists a constant $C(d)$ such that 
    \[W_1(\mu_n, \eta_n) \le C(d) \sup_{\theta \in S^{d-1}} W_1(\mu_{n,\theta}, \eta_{n,\theta}).\]
    
    Let $\{\theta_1, \dots, \theta_N\}$ be an $\epsilon$-net of $S^{d-1}$. For any $\theta$, let $\theta_j$ be its nearest neighbor. Then:
    $$ W_1(\mu_{n,\theta}, \eta_{n,\theta}) \le W_1(\mu_{n,\theta_j}, \eta_{n,\theta_j}) + L_n \epsilon. $$
    Taking the expectation of the supremum,
    \begin{align*}
        \ex [W_1(\mu_n, \eta_n)] &\le C(d) \left( \epsilon \, \ex[L_n] + \sum_{j=1}^N \ex W_1(\mu_{n,\theta_j}, \eta_{n,\theta_j}) \right).
    \end{align*}
    Note that $\ex[L_n] = M_1(\bar{\mu}_n) + M_1(\bar{\eta}_n)$, which is uniformly bounded by assumption. The sum term vanishes as $n \to \infty$ by hypothesis (Condition~2). Taking $n \to \infty$ then $\epsilon \to 0$ yields the result.
\end{proof}

\begin{proof}[Proof of Theorem \ref{thm:vector_result}]
    We apply Proposition \ref{prop:lifting} with $\eta_n = \bar{\mu}_n$.
    Condition 1 holds by assumption (c).
    For Condition 2, fix $\theta \in S^{d-1}$. The projected variables $Z_{i,n} = \langle \theta, Y_{i,n} \rangle$ satisfy the scalar assumptions of Theorem \ref{thm:main_result} (uniform $\Psi_{r_n}$ tails and vanishing moment variance). Note that $\bar{\mu}_{n,\theta} = \ex[\mu_{n,\theta}]$. Thus, Theorem \ref{thm:main_result} implies $\ex W_1(\mu_{n,\theta}, \bar{\mu}_{n,\theta}) \to 0$.
    The conclusion follows immediately.
\end{proof}

\end{document}